\documentclass[twoside,11pt]{article}
\usepackage[hmargin=1in,vmargin=1in]{geometry}

\usepackage[utf8]{inputenc}		
\usepackage{jeffe}

\usepackage[charter]{mathdesign}	
\usepackage[scaled=.96,osf]{XCharter}
\usepackage[scale=0.95]{sourcesanspro}	
\usepackage{textcomp,stmaryrd}

\usepackage{graphicx}
\usepackage{cite}
\usepackage{mathtools}

\newtheorem{lemma}{Lemma}[section]
\newtheorem{theorem}[lemma]{Theorem}

\numberwithin{equation}{section}

\def\Left{\operatorname{\mathit{left}}}		
\def\Right{\operatorname{\mathit{right}}}	
\def\Tail{\operatorname{\mathit{tail}}}		

\definecolor{darkblue}{rgb}{0,0,0.5}
\def\EMPH#1{\textcolor{darkblue}{\textbf{\boldmath \emph{#1}}}}
\def\DIFF#1{\textcolor{BrickRed}{#1}}

\hypersetup{colorlinks=true, urlcolor=Blue, citecolor=Green, linkcolor=Purple, breaklinks, unicode}

\makeatletter
\renewcommand*\env@matrix[1][*\c@MaxMatrixCols c]{%
  \hskip -\arraycolsep
  \let\@ifnextchar\new@ifnextchar
  \array{#1}}
\makeatother

\def\Real{\mathbb{R}}
\def\abs#1{\mathopen| #1 \mathclose|}

\def\arc#1#2{#1\mathord\shortrightarrow#2}

\usepackage{arydshln}
\def\Dx{\Delta\! x}	
\def\Dy{\Delta\! y}	
\def\Torus{\mathbb{T}}

\DeclareRobustCommand{\colvec}{\genfrac(){0pt}{}}
\newcommand{\rowvec}[2]{\left({#1},{#2}\right)}

\allowdisplaybreaks	

\title{A Note on Toroidal Maxwell--Cremona Correspondences}
\author{\href{https://patrickl.in/}{Patrick Lin}
		\\[1ex]
		University of Illinois, Urbana-Champaign}

\date{}

\begin{document}

\maketitle

\begin{abstract}
We explore toroidal analogues of the Maxwell--Cremona correspondence. Erickson and Lin~\cite{el-tmcdc-20} showed the following correspondence for geodesic torus graphs $G$: a \emph{positive} equilibrium stress for $G$, an \emph{orthogonal embedding} of its dual graph $G^*$, and vertex weights such that $G$ is the intrinsic weighted Delaunay graph of its vertices. We extend their results to equilibrium stresses that are not necessarily positive, which correspond to orthogonal drawings of $G^*$ that are not necessarily embeddings. We also give a correspondence between equilibrium stresses and \emph{parallel} drawings of the dual.
\end{abstract}

\section{Introduction}

The Maxwell--Cremona correspondence on the plane establishes an equivalence between the following structures:

\begin{itemize}
\item An \emph{equilibrium stress} on $G$ is an assignment of non-zero weights to the edges of $G$, such that the weighted edge vectors around every interior vertex $p$ sum to zero:
\[
	\sum_{p \colon pq \in E} \omega_{pq} (p-q) = \colvec{0}{0}
\]
\item
A \emph{(orthogonal) reciprocal diagram} for  $G$ is a straight-line drawing of the dual graph $G^*$, in which every edge~$e^*$ is orthogonal to the corresponding primal edge $e$.
\end{itemize}

Erickson and Lin~\cite{el-tmcdc-20} considered graphs on the torus, and established a weaker correspondence for torus graphs: Every \emph{positive} equilibrium graph on any flat torus is affinely equivalent a graph admitting an \emph{embedded} (orthogonal) reciprocal diagram on \emph{some} flat torus. Erickson and Lin's focus on positive equilibria was motivated by equivalence between $G$ admitting an orthogonally embedded dual and $G$ being a \emph{coherent subdivision} (or \emph{weighted Delaunay complex}) of its point set.

Our first result generalizes the correspondence to non-positive equilibria: Every (not necessarily positive) equilibrium graph on any flat torus is affinely equivalent a graph admitting an (not necessarily embedded) orthogonal reciprocal diagram on \emph{some} flat torus.

Maxwell~\cite{m-rfdf-64,m-atrpf-67,m-rffdf-70} derived his reciprocal diagrams with dual edges drawn orthogonally to their corresponding primal edges so as to obtain an equivalence to a third structure, namely, a polyhedral lifting of $G$ in $\R^3$ so that $G$ is the projection of this polyhedral lift, and $G^*$ is the projection of its polar dual around the unit paraboloid. But both Rankine~\cite{r-mam-58,r-pepf-64} beforehand and Cremona~\cite{c-frnsg-72,c-gs-72} afterwards drew their reciprocal diagrams such that dual edges were \emph{parallel} to their corresponding primal edges; indeed Cremona derived parallel reciprocal diagrams via what Crapo~\cite{c-sr-79} called ``skew polarity.'' This skew polar can be obtained by taking the standard polar dual through the unit paraboloid, a quarter turn around the $z$-axis, and then a reflection through the $x\!y$-plane.

On the plane there is no functional difference between a reciprocal diagram in which every edge~$e^*$ is orthogonal to the corresponding primal edge $e$, and one in which every edge $e^*$ is parallel to the corresponding primal edge; indeed, one is just a rotation of the other.

On the torus, however, we find a noticeable difference between orthogonal reciprocal diagrams and parallel reciprocal diagrams. Our second result is a necessary and sufficient condition for an equilibrium torus graph $G$ to admit a parallel reciprocal diagram.

\section{Definitions}

We assume familiarity with the notations and conventions in Erickson and Lin~\cite{el-tmcdc-20}. We define new concepts and recall particularly important ones below. Major differences from Erickson and Lin are \DIFF{highlighted in red}.

\subsection{Drawings}

A \EMPH{drawing} of a graph $G$ on a torus $\Torus$ is a continuous function from $G$ as a topological space to $\Torus$. An \EMPH{embedding} is an injective drawing, mapping vertices of G to distinct points and edges to interior-disjoint simple paths between their endpoints. The faces of an embedding are the components of the complement of the image of the graph; in particular, embeddings are cellular, i.e., all faces are open disks. We will refer to any \EMPH{drawing} of a graph $G$ \DIFF{that is \emph{homotopic} to an embedding of $G$} as a \EMPH{torus graph}.

In any embedded graph, \EMPH{$\Left(d)$} and \EMPH{$\Right(d)$} denote the faces immediately to the left and right of any dart~$d$. Similarly, in any drawing homotopic to the embedding, \EMPH{$\Left(d)$} and \EMPH{$\Right(d)$} refer to the \emph{cycles} bounding the corresponding faces in the embedding.

\subsection{Rotated Duality and Parallel Reciprocality}

Conventionally, the dual graph $G^*$ of an embedded torus graph has the following natural embedding: every vertex $f^*$ of $G^*$ lies in the interior of the corresponding face $f$ of $G$, each edge $e^*$ of $G^*$ crosses only the corresponding edge $e$ of $G$, and each face $p^*$ of $G^*$ contains exactly one vertex $p$ of $G$ in its interior.

\DIFF{We regard any drawing of $G^*$ on $\Torus_\square$ to be \EMPH{rotated dual} to $G$ if its image is homotopic to the clockwise rotation by a quarter circle of the image of the aforementioned natural embedding of $G^*$ on~$\Torus_\square$. More generally, a drawing of $G^*$ is rotated dual to $G$ on some flat torus $\Torus$ if the image of $G^*$ on~$\Torus_\square$ is rotated dual to the image of $G$ on $\Torus_\square$.}

Geodesic graphs $G$ and $G^*$ that are dual to each other are \EMPH{orthogonal reciprocal} if every edge $e$ in~$G$ is orthogonal to its dual edge $e^*$ in $G^*$. \DIFF{We emphasize the $G$ and $G^*$ need \emph{not} be embedded.} An equilibrum stress $\omega$ is an \EMPH{orthogonal reciprocal stress} for $G$ if there is an orthogonal reciprocal \DIFF{drawing} of its dual $G^*$ on the \emph{same} flat torus so that $\omega_e = \abs{e^*}/\abs{e}$ for each edge~$e$.

\DIFF{A geodesic graph $G$ and a geodesic rotated dual $G^*$ are \EMPH{parallel reciprocal} if every edge $e$ in $G$ is parallel to its dual edge $e^*$ in $G^*$. An equilibrium stress $\omega$ is a \EMPH{parallel reciprocal stress} for $G$ if there is a parallel reciprocal graph $G^*$ on the same flat torus so that $\omega_e = \abs{e^*}/\abs{e}$ for each edge $e$.}

\subsection{Circulations and Cocirculations}

We make frequent use of the following lemmas, whose proofs can be found in the paper of Erickson and Lin~\cite{el-tmcdc-20}.

\begin{lemma}[{\cite[Lemma 2.1]{el-tmcdc-20}}]
\label{L:harmonic}
Fix a geodesic drawing of a graph $G$ on $\Torus_\square$ with displacement matrix $\Delta$.
For any circulation~$\phi$ in $G$, we have $\Delta\phi = \Lambda\phi = [\phi]$.
\end{lemma}

\begin{lemma}[{\cite[Lemma 2.4]{el-tmcdc-20}}]
\label{L:embed-on-square}
Fix an essentially simple, essentially 3-connected graph $G$ on $\Torus_\square$, a $2 \times E$ matrix $\Delta$, and a positive stress vector $\omega$.  Suppose for every directed cycle (and therefore any circulation) $\phi$ in $G$, we have $\Delta\phi = \Lambda\phi = [\phi]$.  Then $\Delta$ is the displacement matrix of a geodesic \textbf{drawing} on $\Torus_\square$ that is homotopic to $G$.  If in addition $\omega$ is a positive equilibrium stress for that drawing, the drawing is an embedding.
\end{lemma}

Let $\Lambda$ be the $2 \times E$ matrix whose columns are homology vectors of edges in $G$. Let $\lambda_1$ and $\lambda_2$ denote the first and second rows, respectively, of $\Lambda$.

\begin{lemma}[{\cite[Lemma 2.5]{el-tmcdc-20}}]
\label{L:cocirc}
The row vectors $\lambda_1$ and $\lambda_2$ describe cocirculations in $G$ with cohomology classes $[\lambda_1]^* = \rowvec{0}{1}$ and $[\lambda_2]^* = \rowvec{-1}{0}$.
\end{lemma}

The following lemma forms the analogue to Lemma~\ref{L:cocirc}, for rotated duals.

\begin{lemma}
\label{L:cocirc2}
The row vectors $\lambda_1$ and $\lambda_2$ describe \emph{rotated} cocirculations in $G$ with cohomology classes $[\lambda_1]^* = \rowvec{1}{0}$ and $[\lambda_2]^* = \rowvec{0}{1}$.
\end{lemma}

\begin{proof}
Without loss of generality assume $G$ is embedded on the flat square torus $\Torus_\square$. Let $\gamma_1$ and $\gamma_2$ denote directed cycles in $\Torus_\square$ (\emph{not} on $G$) induced by the boundary edges of $\square$, oriented respectively rightward and upward.

Let $d_0,d_1,\ldots,d_{k-1}$ be the sequence of darts in $G$ that cross $\gamma_2$ from left to right, indexed by the upward order of their intersection points. \DIFF{As shown in Erickson and Lin~\cite{el-tmcdc-20}, these darts dualize in the to a closed walk $d_0^*,d_1^*,\dots,d_{k-1}^*$ in $G^*$ that, in the natural embedding, can be continuously deformed to~$\gamma_2$; when rotated clockwise by a quarter circle, this closed walk can instead be continuously deformed to~$\gamma_1$, so it has the same homology class as $\gamma_1$, i.e., $[\lambda_1]^* = \rowvec{1}{0}$. See Figure~\ref{F:cocirc}.}

\begin{figure}[ht]
\centering
\raisebox{-0.45\height}{\includegraphics[scale=0.45,page=1]{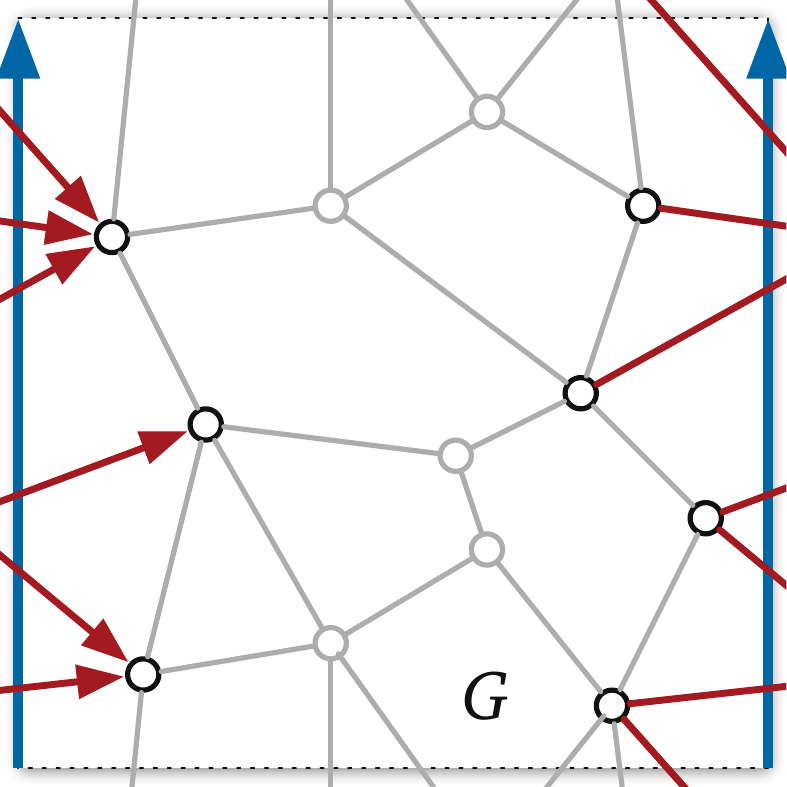}}\quad
$\overset{\text{\normalsize dualize}}{\xrightarrow{\hspace*{3em}}}$\quad
\raisebox{-0.45\height}{\includegraphics[scale=0.45,page=2]{Fig/cohomology}}\quad
$\overset{\text{\normalsize rotate}}{\xrightarrow{\hspace*{3em}}}$\quad
\raisebox{-0.45\height}{\includegraphics[scale=0.45]{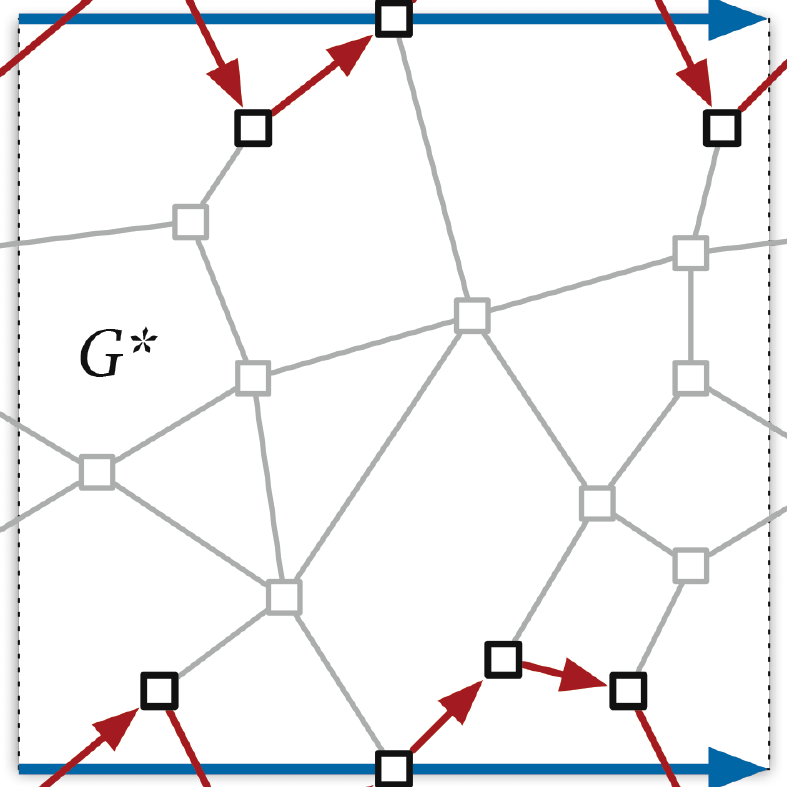}}
\caption{Proof of Lemma \ref{L:cocirc2}: The darts in $G$ crossing either boundary edge of the fundamental square dualize to a closed walk in~$G^*$ parallel to \DIFF{the rotation of} that boundary edge.}
\label{F:cocirc}
\end{figure}

Symmetrically, the darts crossing $\gamma_1$ upward define a closed walk in $G^*$ with the same homology class as \DIFF{$\gamma_2$, so $[\lambda_2]^* = \rowvec{0}{1}$}.
\end{proof}

\section{Orthogonal Reciprocality}
\label{S:oreciprocal}

For an essentially simple, essentially $3$-connected geodesic graph $G$ on the square flat torus $\Torus_\square$ and an equilibrium stress $\omega$, we let $\Delta$ be the $2 \times E$ displacement matrix for $G$, and let $\Omega$ be the $E\times E$ matrix whose diagonal entries are $\Omega_{e,e} = \omega_e$ and whose off-diagonal entries are all~$0$. Define
\begin{equation}
	\alpha = \sum_e \omega_e \Dx_e^2,
	\qquad \beta = \sum_e \omega_e \Dy_e^2,
	\qquad \gamma = \sum_e \omega_e \Dx_e\Dy_e.
	\label{Eq:alphabet}
\end{equation}
These values are the entries of the covariance matrix $\Delta\Omega\Delta^T = \begin{psmallmatrix} \alpha & \gamma \\ \gamma & \beta \end{psmallmatrix}$.

Section~5 of Erickson and Lin~\cite{el-tmcdc-20} consider the case of embedded geodesic torus graphs with positive equilibrium stresses. It turns out that their analysis also applies to the case of a drawing of a torus graph~$G$ that is not necessarily an embedding and the stress is not necessarily positive, with minimal changes. For completeness we give the relevant restatements of the main results:

\begin{theorem}
\label{T:skew-rec-nonlinear}
Let $G$ be a geodesic \DIFF{drawing} on $\Torus_\square$ \DIFF{homotopic to an embedding} with a \DIFF{(not necessarily positive)} equilibrium stress $\omega$.  Let $\alpha$, $\beta$, and~$\gamma$ be defined as in~Equation \eqref{Eq:alphabet}.
If $\alpha\beta - \gamma^2 = 1$, then $\omega$ is an \DIFF{orthogonal} reciprocal stress for the image of $G$ on~$\Torus_M$ if and only if $M = \sigma R \begin{psmallmatrix} \beta & -\gamma \\ 0 & 1 \end{psmallmatrix}$ for any rotation matrix $R$ and any real number $\sigma>0$.  On the other hand, if $\alpha\beta - \gamma^2 \ne 1$, then $\omega$ is not a reciprocal stress for the image of $G$ on any flat torus. 
\end{theorem}

The result can be reinterpreted in terms of force diagrams as follows:

\begin{lemma}
\label{L:forceM}
Let $G$ be a geodesic \DIFF{drawing} on $\Torus_M$ \DIFF{homotopic to an embedding}, and let $\omega$ be a \DIFF{(not necessarily positive)} equilibrium stress for $G$.  The \DIFF{orthogonal} force diagram of $G$ with respect to $\omega$ lies on the flat torus $\Torus_N$, where $N = J M \Delta\Omega\Delta^T J^T$.
\end{lemma}

The main differences between this setting and the setting of positive stress vector $\omega$ (and thus embedded $G$) are as follows.

When $\omega$ is positive, then $\alpha\beta-\gamma^2
	=
	\frac{1}{2}\sum_{e,e'} \omega_e \omega_{e'}
			\big|
				\begin{smallmatrix}
				\Dx_{e\;} & \Dy_{e\;} \\
				\Dx_{e'} & \Dy_{e'}
				\end{smallmatrix}
			\big|^2
	> 0$,
so in fact the requirement $\alpha\beta - \gamma^2 = 1$ is just a scaling condition: given any positive stress vector $\omega$, the stress vector $\omega/\sqrt{\alpha\beta-\gamma^2}$ is a positive stress vector that satisfies said requirement. If $\omega$ is non-positive, however, it is possible that $\alpha\beta-\gamma^2 < 0$, in which case no scaling of $\omega$ is an orthogonal reciprocal stress on any flat torus (in terms of force diagrams, no scaling of $\omega$ can result in $M$ being equal to $N$).

Furthermore, when $\omega$ is positive, the force diagram is embedded on $\Torus_N$, and every face of the force diagram is a convex polygon; if $\omega$ is not necessarily positive, then the force diagram still lies on $\Torus_N$, but is not necessarily embedded, and faces may self-intersect.

\subsection{Examples}

Consider the symmetric embedding of $K_7$ shown in Figure~\ref{F:weird}. The edges fall into one of three equivalence classes, with slopes $3$, $2/3$, $-1/2$ and lengths $\sqrt{10}/7$, $\sqrt{5}/7$, $\sqrt{14}/7$, respectively. Assigning the edges of slope $3$ a stress of $2$, the edges of slope $2/3$ a stress of $-1$, and the edges of slope $-1/2$ a stress of $3$, we can verify that this indeed induces an equilibrium stress, and furthermore $\alpha\beta-\gamma^2 = 1$, so we can find a $2 \times 2$ matrix $M$ such that the image of $G$ on $\Torus_M$ has an orthogonal reciprocal diagram. However, this reciprocal diagram has coincident vertices and overlapping edges and self-intersecting faces; see Figure~\ref{F:weirdrec}.

\begin{figure}[ht]
\centering
\includegraphics[scale=0.5,page=3]{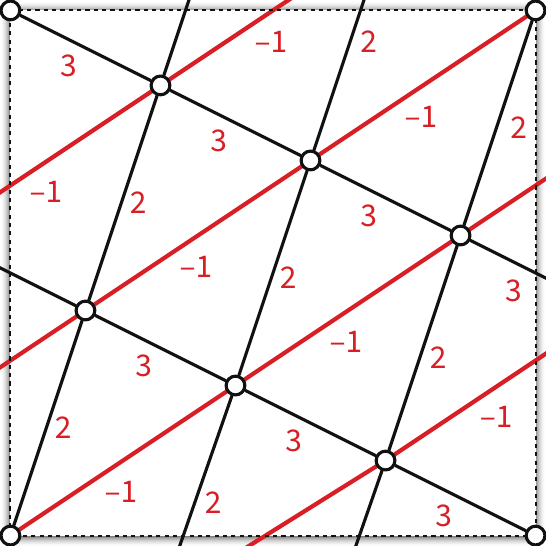}
\caption{The symmetric embedding of $K_7$ with weights $-1$, $2$, and $3$.}
\label{F:weird}
\end{figure}

\begin{figure}[ht]
\centering
\begin{tabular}{cc}
\raisebox{-0.5\height}{\includegraphics[scale=0.33,page=2]{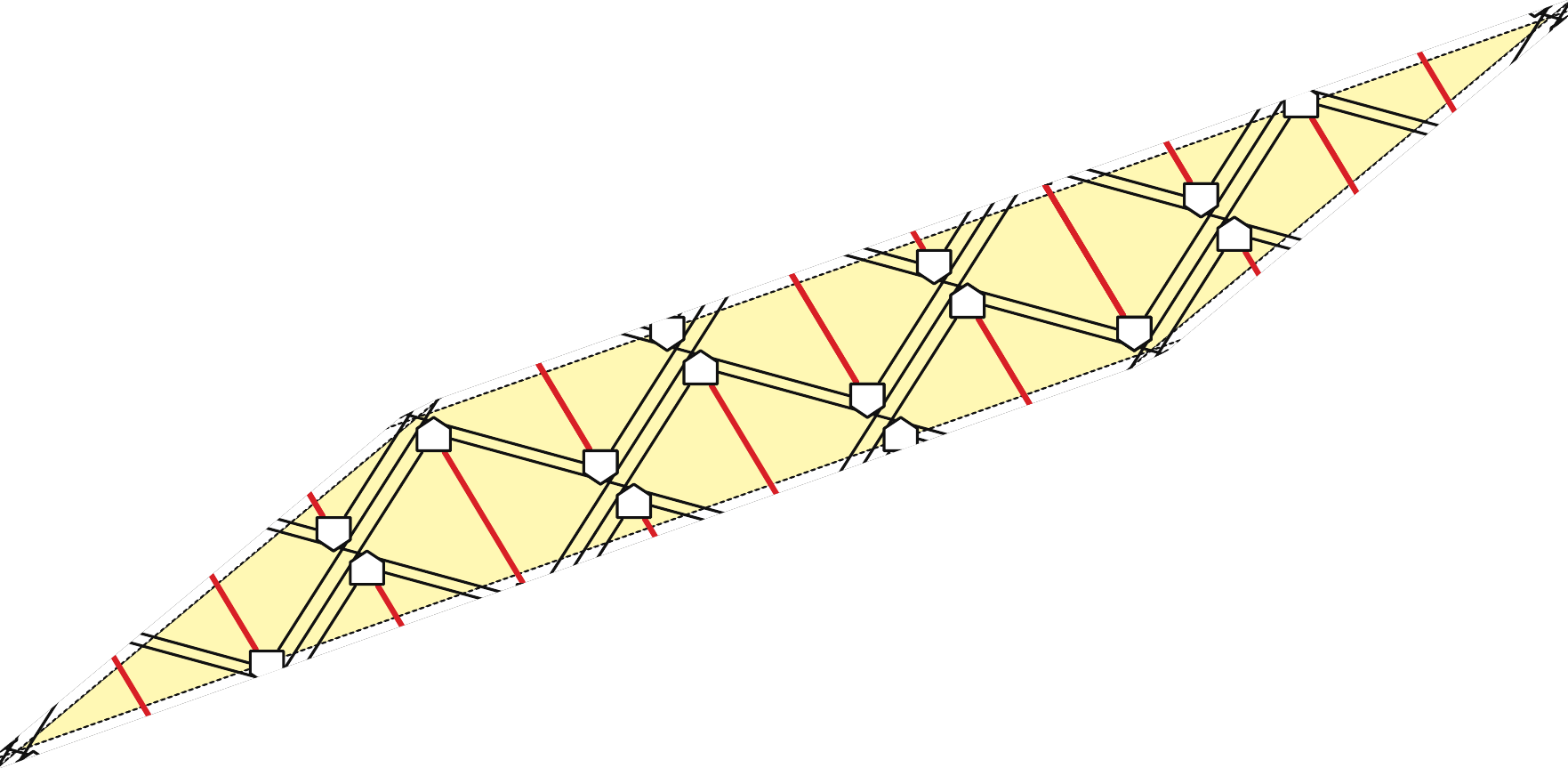}} &
\raisebox{-0.5\height}{\includegraphics[scale=0.33,page=4]{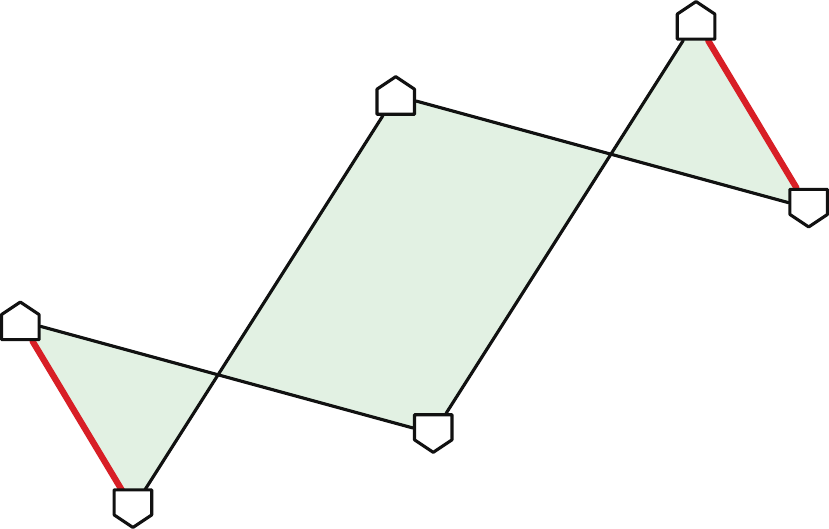}} \\
\small (a) & \small (b)
\end{tabular}
\caption{(a) The reciprocal diagram of the image of $K_7$ on $\Torus_M$. Overlapping vertices and edges are drawn as being distinct, but vertices close to and pointing to each other actually occupy the same position. (b) One of the faces of the reciprocal diagram.}
\label{F:weirdrec}
\end{figure}

If we instead assign the edges of slope $3$ and $2/3$ a stress of $1$, and the edges of slope $-1/2$ a stress of $-1$, then this is an equilibrium stress, but $\alpha\beta - \gamma^2 = -1$, so no scaling of this stress vector can be a reciprocal stress for $G$ on any flat torus.

\section{Parallel Reciprocality}
\label{S:preciprocal}

In this section we give an analogous version of the previous section, but for parallel reciprocality. Major differences from Erickson and Lin~\cite{el-tmcdc-20} are \DIFF{highlighted in red}.

Fix an essentially simple, essentially 3-connected geodesic graph $G$ on the \emph{square} flat torus~$\Torus_\square$, along with a \DIFF{(not necessarily positive)} equilibrium stress $\omega$ for $G$.  In this section, we describe simple necessary and sufficient conditions for $\omega$ to be a \DIFF{parallel} reciprocal stress for $G$.

Let $\Delta$ be the $2 \times E$ displacement matrix of $G$, and let $\Omega$ be the $E\times E$ matrix whose diagonal entries are $\Omega_{e,e} = \omega_e$ and whose off-diagonal entries are all~$0$.  The results in this section are phrased in terms of the covariance matrix $\Delta\Omega\Delta^T$.

\subsection{The Square Flat Torus}

We first establish necessary and sufficient conditions for $\omega$ to be a \DIFF{parallel} reciprocal stress for $G$ on the \emph{square} flat torus $\Torus_\square$, in terms of the covariance matrix $\Delta\Omega\Delta^T$.

\begin{lemma}
\label{L:square-abg}
If $\omega$ is a \DIFF{parallel} reciprocal stress for $G$ on $\Torus_\square$, then $\Delta\Omega\Delta^T = \DIFF{I}$.
\end{lemma}

\begin{proof}
Suppose $\omega$ is a \DIFF{parallel} reciprocal stress for $G$ on $\Torus_\square$.  Then there is a geodesic
\DIFF{drawing} of the dual graph $G^*$ on $\Torus_\square$ where $e~\DIFF{\parallel}~e^*$ and 
$\abs{e^*} = \omega_e\abs{e}$ for every edge $e$ of $G$.  Let 
$\Delta^* = (\Delta\Omega)^{\DIFF{T}}$ denote the~$E \times 2$ matrix whose rows are the
displacement row vectors of $G^*$.

Recall from Lemma~\ref{L:cocirc} that the first and second rows of $\Lambda$ describe cocirculations of $G$ with cohomology classes \DIFF{$\rowvec{1}{0}$} and \DIFF{$\rowvec{0}{1}$}, respectively.  Applying Lemma~\ref{L:harmonic} to $G^*$ implies $\theta\Delta^* = [\theta]^*$ for any cocirculation $\theta$ in $G$.  It follows immediately that
$\Lambda\Delta^* = \DIFF{\begin{psmallmatrix} 1 & 0 \\ 0 & 1 \end{psmallmatrix} = I}$.

Because the rows of $\Delta^*$ are the displacement vectors of $G^*$, for every vertex $p$ of $G$ we have
\begin{equation}
	\sum_{q\colon pq\in E} \Delta^*_{(\arc{p}{q})^*}
	~=~
	\sum_{d \colon \Tail(d) = p} \Delta^*_{d^*}
	~=~
	\sum_{d \colon \Left(d^*) = p^*} \Delta^*_{d^*}
	~=~
	\rowvec{0}{0}.
	\label{Eq:circulation}
\end{equation}
It follows that the \emph{columns} of $\Delta^*$ describe circulations in $G$.
Lemma~\ref{L:harmonic} now implies that $\Delta\Delta^* = \DIFF{\Delta\Omega\Delta^T = I}$.
\end{proof}

\begin{lemma}
\label{L:dual-embed-on-square}
Fix an $E \times 2$ matrix $\Delta^*$.  If $\Lambda\Delta^* = \DIFF{I}$, then $\Delta^*$ is the displacement matrix of a geodesic drawing on $\Torus_\square$ that is dual to $G$.  Moreover, if that drawing has a \DIFF{positive} equilibrium stress, it is actually an embedding.
\end{lemma}

\begin{proof}
Let $\lambda_1$ and $\lambda_2$ denote the rows of $\Lambda$.  Rewriting the identity
$\Lambda\Delta^* = \DIFF{I}$ in terms of these row vectors gives us
\(
	\sum_e \Delta^*_e\lambda_{1,e} = \DIFF{\rowvec{1}{0}} = [\lambda_1]^*
\)
\text{and} 
\(
	\sum_e \Delta^*_e\lambda_{2,e} = \DIFF{\rowvec{0}{1}} = [\lambda_2]^*.
\)
Extending by linearity, we have $\sum_e \Delta^*_e\theta_e = [\theta]^*$ for every cocirculation~$\theta$ in~$G^*$.  The result now follows from Lemma~\ref{L:embed-on-square}. 
\end{proof}

\begin{lemma}
\label{L:abg-square}
If $\Delta\Omega\Delta^T = \DIFF{I}$, then $\omega$ is a \DIFF{parallel} reciprocal stress for $G$ on $\Torus_\square$. \DIFF{If $\omega$ is a positive equilibrium stress, then the parallel reciprocal diagram is in fact \emph{embedded} on $\Torus_\square$.}
\end{lemma}

\begin{proof}
Set $\Delta^* = \DIFF{(\Delta\Omega)^T}$.  Because $\omega$ is an equilibrium stress in $G$, for every vertex $p$ of $G$ we have
\begin{equation}
	\sum_{q\colon pq\in E} \Delta^*_{(\arc{p}{q})^*}
	~=~
	\sum_{q\colon pq\in E} \omega_{pq} \Delta_{\arc{p}{q}}^{\DIFF{T}}
	~=~
	\rowvec{0}{0}.
	\label{Eq:circulation2}
\end{equation}
It follows that the columns of $\Delta^*$ describe circulations in~$G$, and therefore Lemma~\ref{L:harmonic} implies $\Lambda\Delta^* = \Delta\Delta^* = \DIFF{\Delta (\Delta\Omega)^T = \Delta\Omega\Delta^T = I}$.

Lemma \ref{L:dual-embed-on-square} now implies that~$\Delta^*$ is the displacement matrix of a drawing~$G^*$ dual to $G$.  Moreover, the stress vector $\omega^*$ defined by $\omega^*_{e^*} = 1/\omega_e$ is an equilibrium stress for $G^*$: under this stress vector, the darts leaving any dual vertex~$f^*$ are dual to the clockwise boundary cycle of face $f$ in $G$.  Thus \DIFF{if $\omega$ is positive, then} $G^*$ is in fact an embedding.  By construction, each edge of $G^*$ is \DIFF{parallel} to the corresponding edge of~$G$.
\end{proof}

\subsection{Force Diagrams}
\label{SS:force}

The results of the previous section have a more physical interpretation that may be more intuitive.  Let~$G$ be any geodesic graph on the unit square flat torus $\Torus_\square$.  Recall that any \DIFF{equilibrium} stress~$\omega$ on $G$ induces an \DIFF{equilibrium} stress on its universal cover $\widetilde{G}$, which in turn induces a \DIFF{(parallel)} reciprocal diagram~$(\widetilde{G})^*$ by the classical Maxwell--Cremona correspondence.  This infinite plane graph $(\widetilde{G})^*$ is doubly-periodic, but in general with a different period lattice from the universal cover $\widetilde{G}$.

Said differently, we can always construct another geodesic torus graph $H$ that is combinatorially dual to $G$, such that for every edge $e$ of $G$, the corresponding edge $e^*$ of $H$ is \DIFF{parallel} to $e$ and has length~$\omega_e\cdot\abs{e}$; however, this torus graph $H$ does not necessarily lie on the \emph{square} flat torus.  (By construction, $H$ is the unique torus graph whose universal cover is $(\widetilde{G})^*$, the reciprocal diagram of the universal cover of $G$.)  We call $H$ the \EMPH{\DIFF{parallel} force diagram} of $G$ with respect to $\omega$.  The \DIFF{parallel} force diagram $H$ lies on the same flat torus $\Torus_\square$ as $G$ if and only if $\omega$ is a \DIFF{parallel} reciprocal stress for $G$.

\begin{lemma}
\label{L:force}
Let $G$ be a geodesic graph in $\Torus_\square$, and let $\omega$ be a \DIFF{(not necessarily positive)} equilibrium stress for $G$.  The \DIFF{parallel} force diagram of $G$ with respect to $\omega$ lies on the flat torus $\Torus_M$, where $M = \DIFF{\Delta\Omega\Delta^T}$.
\end{lemma}

\begin{proof}
As usual, let $\Delta$ be the displacement matrix of $G$.  Let $\Delta^*$ denote the displacement matrix of the force diagram $H$; by definition, we have $\Delta^* = \DIFF{(\Delta\Omega)^T = \Omega\Delta^T}$.   Equation~\eqref{Eq:circulation2} implies that the columns of $\Delta^*$ are circulations in $G$.  Thus, Lemma~\ref{L:harmonic} implies that $\Lambda\Delta^* = \Delta\Delta^* = \DIFF{\Delta\Omega\Delta^T}$.

Set $M = \DIFF{\Delta\Delta^* = \Delta\Omega\Delta^T}$.  We immediately have $\Lambda\Delta^* = \DIFF{M = M^T}$ and therefore $\Lambda\Delta^*(M^T)^{-1} = \DIFF{I}$.  Lemma~\ref{L:dual-embed-on-square} implies that $\Delta^*(M^T)^{-1}$ is the displacement matrix of a homotopic \DIFF{drawing} of $G^*$ on $\Torus_\square$.  It follows that $\Delta^*$ is the displacement matrix of the image of $G^*$ on~$\Torus_M$.  We conclude that $H$ is a translation of the image of $G^*$ on $\Torus_M$.
\end{proof}

\subsection{Arbitrary Flat Tori}

\DIFF{In the case of \emph{orthogonal} reciprocal diagrams, Erickson and Lin~\cite{el-tmcdc-20} established necessary and sufficient conditions for an equilibrium graph to be reciprocal on \emph{some} flat torus.}

\DIFF{In contrast, we find that in the case of \emph{parallel} reciprocal diagrams, an equilibrium graph is reciprocal on $\Torus_\square$ if and only if it is reciprocal on \emph{every} flat torus. This is perhaps unsurprising: orthogonality relies on the \emph{conformal} structure of the flat torus $\Torus$; parallelism is an \emph{affine} property.}

\begin{lemma}
\label{L:nonsq-abg}
If $\omega$ is a \DIFF{parallel} reciprocal stress for a geodesic graph $G$ on $\Torus_M$ for some non-singular matrix $M$, then \DIFF{$\Delta\Omega\Delta^T=I$}.
\end{lemma}

\begin{proof}
Suppose $\omega$ is a \DIFF{parallel} reciprocal stress for $G$ on $\Torus_M$.  Then there is a geodesic \DIFF{drawing} of the dual graph $G^*$ on $\Torus_M$ where $e~\DIFF{\parallel}~e^*$ and $\abs{e^*} = \omega_e\abs{e}$ for every edge $e$ of $G$.

We will consider the geometry of $G$ and $G^*$ on the reference torus~$\Torus_\square$.  (The \DIFF{drawinsg} of $G$ and $G^*$  on the reference torus $\Torus_\square$ are still dual, but not necessarily reciprocal.)  Let~$\Delta$ denote the $2\times E$ \emph{reference} displacement matrix for $G$, whose columns are the displacement vectors for~$G$ on the square torus~$\Torus_\square$.  Then the columns of $M\Delta$ are the \emph{native} displacement vectors for~$G$ on the torus~$\Torus_M$.  Thus, the \emph{native} displacement row vectors of~$G^*$ are given by the rows of the $E \times 2$ matrix $(M\Delta\Omega)^{\DIFF{T}}$.  Finally, let $\Delta^* = (M\Delta\Omega)^{\DIFF{T}} (M^T)^{-1}$ denote the \emph{reference} displacement row vectors for $G^*$ on the square torus $\Torus_\square$.  
We can rewrite this definition as
\begin{equation}
\begin{aligned}
\Delta^*
	&= (M\Delta\Omega)^{\DIFF{T}}(M^T)^{-1}	
			\\
	&= \DIFF{\Omega \Delta^T\, M^T} (M^{T})^{-1} \\
	&~\DIFF{= \Omega\Delta^T.}
\end{aligned}
\label{Eq:Delta-star}
\end{equation}
%

Because the rows of $\Delta^*$ are the displacement vectors for $G^*$, equation \eqref{Eq:circulation} implies that the \emph{columns} of $\Delta^*$ describe circulations in $G$, and therefore $\Delta\Delta^* = \Lambda\Delta^* = \DIFF{\begin{psmallmatrix} 1 & 0 \\ 0 & 1 \end{psmallmatrix} = I}$ by Lemmas~\ref{L:harmonic} and~\ref{L:cocirc}.  We conclude that
\(\Delta\Omega\Delta^T
	= \DIFF{\Delta\Delta^* = I}\).
\end{proof}

\begin{lemma}
\label{L:abg-nonsq}
If \DIFF{$\Delta\Omega\Delta^T = I$}, then $\omega$ is a \DIFF{parallel} reciprocal stress for $G$ on $\Torus_M$ where \DIFF{$M$ is any non-singular $2\times 2$ matrix}. \DIFF{Moreover, if $\omega$ is a \emph{positive} equilibrium stress, then the reciprocal diagram is embedded on $\Torus_M$.}
\end{lemma}

\begin{proof}
Suppose \DIFF{$\Delta\Omega\Delta^T = I$}.  Fix an arbitrary $2\times 2$ \DIFF{non-singular matrix $M$}.  Let~$\Delta$ denote the $2 \times E$ \emph{reference} displacement matrix for $G$ on the square flat torus~$\Torus_\square$, and define the $E \times 2$ matrix $\Delta^* = (M\Delta\Omega)^{\DIFF{T}} (M^T)^{-1}$.

Derivation \eqref{Eq:Delta-star} in the proof of Lemma \ref{L:nonsq-abg} implies $\Delta^* = \DIFF{\Omega \Delta^T}$.
It follows that
\[
	\Delta\Delta^*
	~=~ \DIFF{\Delta \Omega \Delta^T~=~I}.
\]
Because $\omega$ is an equilibrium stress in $G$, for every vertex $p$ of $G$ we have
\begin{equation}
	\sum_{q\colon pq\in E} \Delta^*_{(\arc{p}{q})^*}
	~=~
	\sum_{q\colon pq\in E} \omega_{pq} \Delta_{\arc{p}{q}}^{\DIFF{T}}
	~=~
	\rowvec{0}{0}.
	\label{Eq:circulation3}
\end{equation}
Once again, the columns of $\Delta^*$ describe circulations in $G$, so Lemma~\ref{L:harmonic} implies $\Lambda\Delta^* = \Delta\Delta^* = \DIFF{I}$.  Lemma~\ref{L:dual-embed-on-square} now implies that $\Delta^*$ is the displacement matrix of a homotopic \DIFF{drawing} of $G^*$ on~$\Torus_\square$, \DIFF{and if $\omega$ is positive, said drawing is in fact an embedding}.  It follows that 
$(M\Delta\Omega)^{\DIFF{T}} = \Delta^*M^T$ is the displacement matrix of the image
of $G^*$ on $\Torus_M$.  By construction, each edge of~$G^*$ is \DIFF{parallel} to its corresponding edge of~$G$.  We conclude that $\omega$ is a \DIFF{parallel} reciprocal stress for~$G$.
\end{proof}

Our main theorem now follows immediately. 

\begin{theorem}
\label{T:skew-rec-nonlinear}
Let $G$ be a geodesic graph on $\Torus_\square$ with an equilibrium stress $\omega$. If \DIFF{$\Delta\Omega\Delta^T = I$}, then $\omega$ is a \DIFF{parallel} reciprocal stress for the image of $G$ on~$\Torus_M$ \DIFF{for any non-singular matrix $M$}; \DIFF{furthermore, if $\omega$ is a \emph{positive} equilibrium stress, then the parallel reciprocal diagram is embedded on $\Torus_M$}.  On the other hand, if~\DIFF{$\Delta\Omega\Delta^T \ne I$}, then $\omega$ is not a \DIFF{parallel} reciprocal stress for the image of $G$ on any flat torus.
\end{theorem}

In terms of force diagrams:

\begin{lemma}
\label{L:forceM}
Let $G$ be a geodesic graph on $\Torus_M$, and let $\omega$ be a positive equilibrium stress for $G$.  The \DIFF{parallel} force diagram of $G$ with respect to $\omega$ lies on the flat torus $\Torus_N$, where $N = \DIFF{M \Delta\Omega\Delta^T}$.
\end{lemma}

\begin{proof}
We argue exactly as in the proof of Lemma \ref{L:force}.
Let $\Delta$ be the \emph{reference} displacement matrix of (the image of) $G$ on $\Torus_\square$.  Then the \emph{native} displacement matrix of the force diagram is $\Delta^* = (M\Delta\Omega)^{\DIFF{T}} \DIFF{= \Omega\Delta^TM^T}$.  Equation~\eqref{Eq:circulation3} and Lemma~\ref{L:harmonic} imply that $\Lambda\Delta^* = \DIFF{\Delta\Omega\Delta^TM^T}$.

Now let $N = \DIFF{M \Delta\Omega\Delta^T}$.  We immediately have $\DIFF{N^T} = \Lambda\Delta^*$ and thus $\Lambda\Delta^*(N^T)^{-1} = \DIFF{I}$.
Lemma~\ref{L:dual-embed-on-square} implies that $\Delta^*(N^T)^{-1}$ is the displacement matrix of a homotopic \DIFF{drawing} of $G^*$ on $\Torus_\square$.  It follows that $\Delta^*$ is the displacement matrix of the image of $G^*$ on~$\Torus_N$.
\end{proof}

\subsection{Example}

Consider the symmetric embedding of $K_7$ on the square flat torus $\Torus_\square$ shown in Figure~\ref{F:uniform}.  Symmetry implies that $G$ is in equilibrium with respect to the uniform stress $\omega \equiv 1$.  Straightforward calculation gives us $\Delta\Omega\Delta^T = \begin{psmallmatrix} 2 & 1 \\ 1 & 2 \end{psmallmatrix}$ for this stress vector.  Thus, Lemma~\ref{L:square-abg} immediately implies that $\omega$ is not a \DIFF{parallel} reciprocal stress for $G$; rather, by Lemma~\ref{L:force}, the \DIFF{parallel} force diagram of $G$ with respect to $\omega$  lies on the torus $\Torus_M$, where $M = \DIFF{\Delta\Omega\Delta^T = \begin{psmallmatrix} 2 & 1 \\ 1 & 2 \end{psmallmatrix}}$.

\begin{figure}[ht]
\centering
\includegraphics[scale=0.5,page=1]{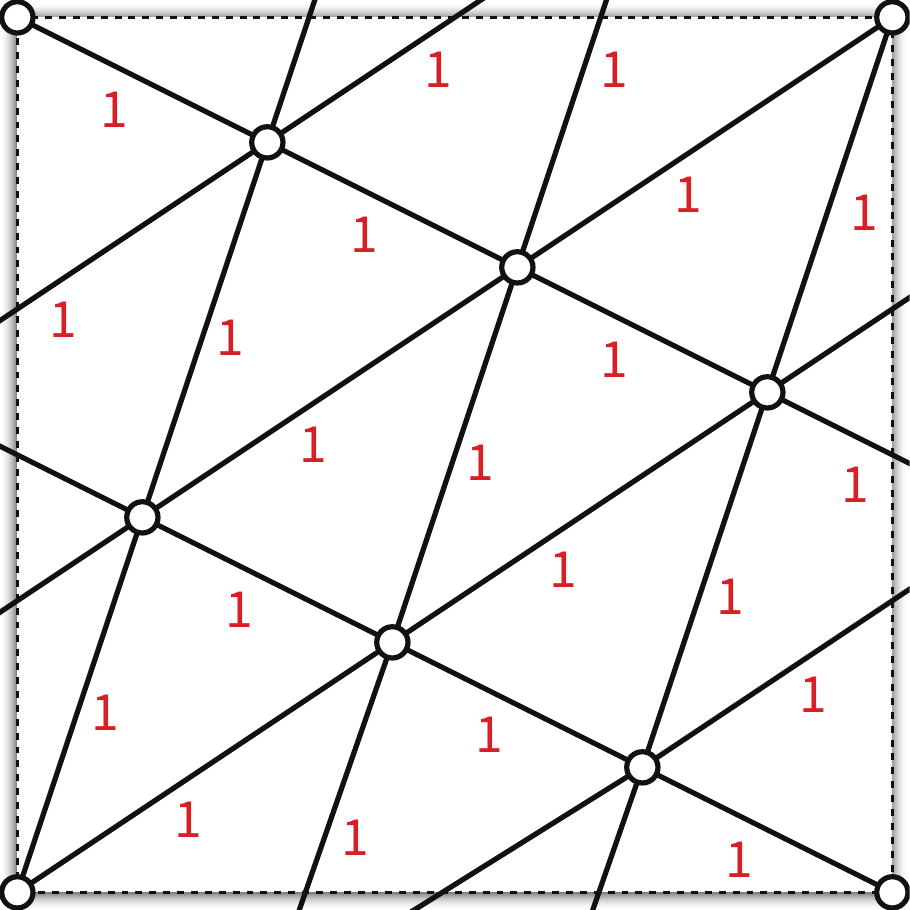}
\caption{The symmetric embedding of $K_7$ with the uniform equilibrium stress $\omega\equiv 1$.}
\label{F:uniform}
\end{figure}

\DIFF{Erickson and Lin~\cite{el-tmcdc-20} show that the scaled uniform stress $\omega' \equiv 1/\sqrt3$ is an orthogonal reciprocal stress for $G$ on the flat torus $\Torus_M$ where $M = \frac{1}{\sqrt{3}} \begin{psmallmatrix} 2 & -1 \\ 0 & \sqrt{3} \end{psmallmatrix}$. On the other hand, Theorem~\ref{T:skew-rec-nonlinear} implies that~$\omega \equiv 1$ and its scalings are \emph{never} parallel reciprocal stresses.}

\section{Conclusions and Open Questions}

By working with orthogonal reciprocality, Erickson and Lin~\cite{el-tmcdc-20} were able to extend the correspondence to coherent subdivisions: a geodesic torus graph $G$ admits an orthogonal reciprocal diagram if and only if it is coherent, i.e., the weighted Delaunay graph of its vertices with respect to some vector of weights. 

Comparing our results with the results of Erickson and Lin, we find that parallel reciprocality and orthogonal reciprocality coincide exactly when the torus is square, and so parallel reciprocality corresponds to coherence on square flat tori. A partial explanation of this may lie in Cremona's skew polarity: a quarter-turn rotation is an automorphism only of \emph{square} flat tori.

We pose as a line of inquiry whether there exists a more general correspondence between parallel reciprocality and some variant of coherence, though as previously noted, parallel reciprocality is an affine property (whereas orthogonal reciprocality is a conformal property), and so any form of coherence corresponding to parallel reciprocality will also need to be an affine property.

\paragraph*{Acknowledgements.}  We thank Jeff Erickson for helpful comments, especially about skew polarity, and for help with generating figures.

\bibliographystyle{bib/newuser-doi}
\bibliography{mctorus-ext}

\end{document}